\documentclass[11pt]{article}
\usepackage[dvips]{graphicx}
\usepackage{amssymb}
\usepackage{amsmath}
\usepackage{amsthm}
\usepackage{latexsym}
\usepackage[english]{babel}
\usepackage{appendix}

\usepackage[usenames,dvipsnames]{color}

\newtheorem{theorem}{Theorem}[section]
\newtheorem{definition}{Definition}[section]
\newtheorem{lemma}{Lemma}[section]
\newtheorem{proposition}{Proposition}[section]

\newtheorem{corollary}{Corollary}[section]

\def \C {\mathbb{C}}
\def \R {\mathbb{R}}
\def \H {\mathbb{H}}
\def \N {\mathbb{N}}
\def \S {\mathcal{S}}

\def \P {\mathcal{P}}

\def \bP {\overline{P}^{\prime}}
\def \bQ {\overline{Q}^{\prime}}

\def\Xint#1{\mathchoice
{\XXint\displaystyle\textstyle{#1}}%
{\XXint\textstyle\scriptstyle{#1}}%
{\XXint\scriptstyle\scriptscriptstyle{#1}}%
{\XXint\scriptscriptstyle\scriptscriptstyle{#1}}%
\!\int}
\def\XXint#1#2#3{{\setbox0=\hbox{$#1{#2#3}{\int}$ }
\vcenter{\hbox{$#2#3$ }}\kern-.6\wd0}}

\def\dashint{\Xint-}

\begin{document}

\title{$\overline{Q}'$-curvature flow on Pseudo-Einstein CR manifolds}

\author{Ali Maalaoui$^{(1)}$ \& Vittorio Martino$^{(2)}$}
\addtocounter{footnote}{1}
\footnotetext{Department of Mathematics, Clark University, Worcester, MA 01610, USA.\\
E-mail address: {\tt{amaalaoui@clarku.edu}}}
\addtocounter{footnote}{1}
\footnotetext{Dipartimento di Matematica, Universit\`a di Bologna, piazza di Porta S.Donato 5, 40126 Bologna, Italy.\\
E-mail address: {\tt{vittorio.martino3@unibo.it}}}

\date{}
\maketitle

\vspace{5mm}

{\noindent\bf Abstract} {\small In this paper we consider the problem of prescribing the $\overline{Q}'$-curvature on three dimensional Pseudo-Einstein CR manifolds. We study the gradient flow generated by the related functional and we will prove its convergence to a limit function under suitable assumptions.}

\vspace{5mm}


\vspace{4mm}


\vspace{4mm}


\section{Introduction and statement of the results}

\noindent
Let $(M,T^{1,0}M,\theta)$ be a CR three manifold, which we will always assume smooth and closed. It is known that one can construct a pair $(Q,P_{\theta})$ such that under a conformal change of the contact form $\hat{\theta}=e^{2u}\theta$, one has
$$P_{\theta}u+Q_{\theta}=Q_{\hat{\theta}}e^{4u}$$
where the Paneitz operator $P_{\theta}=(\Delta_{b})^{2}+T^{2}+l.o.t.$; in particular the operator $P_{\theta}$ contains the space of CR pluriharmonic functions $\mathcal{P}$ in its kernel, moreover the total $Q$-curvature is always zero \cite{Hir2}, hence it does not provide any extra geometric information.\\
Therefore, one considers another pair $(P^{\prime},Q^{\prime})$, see \cite{Bran}, where $P^{\prime}$ is a Paneitz type operator satisfying $P^{\prime}=4(\Delta_{b})^{2}+l.o.t.$ and is defined on the space of pluriharmonic functions and the $Q^{\prime}$-curvature is defined implicitly so that
$$P^{\prime}_{\theta}u+Q^{\prime}_{\theta}-\frac{1}{2}P_{\theta}(u^{2})=Q^{\prime}_{\hat{\theta}}e^{2u},$$
which is equivalent to
\begin{equation}\label{eq1}
P^{\prime}_{\theta}u+Q^{\prime}_{\theta}=Q^{\prime}_{\hat{\theta}}e^{2u}\text{ mod }\mathcal{P}^{\perp}.
\end{equation}
In the case of pseudo-Einstein three dimensional CR manifolds (we refer the reader to the next section for further details), in \cite{CaYa} the authors showed that the total $Q^{\prime}$-curvature is not always zero and it is invariant under the conformal change of the contact structure; in particular it is proportional to the Burns-Epstein invariant $\mu(M)$ (see \cite{BE}, \cite{CL}) and if $(M,J)$ is the boundary of a strictly pseudo-convex domain $X$, then
$$\int_{M}Q^{\prime}\theta\wedge d\theta=16\pi^{2}\Big(\chi(X)-\int_{X}(c_{2}-\frac{1}{3}c_{1}^{2})\Big),$$
where $c_{1}$ and $c_{2}$ are the first and second Chern forms of the K\"{a}hler-Einstein metric on $X$ obtained by solving Fefferman's equation.\\
At this point, in order to avoid the problem of solving orthogonally to the infinite dimensional space $\mathcal{P}^{\perp}$, in \cite{CaYa2} it is introduced on pseudo-Einstein three dimensional CR manifolds a new couple $(\bP, \bQ)$ which comes from the projection of equation $(\ref{eq1})$ on to the space of $L^{2}$ CR  pluriharmonic functions $\hat{\mathcal{P}}$, the completion of $\P$ under the $L^{2}$-norm. Since the $P^{\prime}$-operator is only defined after projection on $\mathcal{P}$, if $\Gamma: L^{2}(M)\to \hat{\mathcal{P}}$ denotes the orthogonal projection and we let $\bP=\Gamma\circ P^{\prime}$ and $\bQ=\Gamma \circ Q^{\prime}$, then one can consider the problem of prescribing the $\bQ$-curvature, under conformal change of the contact structure on pseudo-Einstein CR manifolds. In particular, for a given a function $f\in \hat{\mathcal{P}}$, one wants to solve the following equation
\begin{equation}\label{eq2}
P' u+Q'=fe^{2u} \text{ mod } \mathcal{P}^{\perp},
\end{equation}
that is equivalent to
$$\bP u+\bQ =\Gamma(fe^{2u}).$$
Therefore, if $u$ solves $(\ref{eq2})$, then for $\tilde{\theta}=e^{u}\theta$, one has $\bQ_{\tilde{\theta}}=f$. Let us explicitly notice the differences between the two projections, since the space of $L^{2}$ CR pluriharmonic functions $\hat{\mathcal{P}}$ does not depend on the contact form. Thus, $\bQ$ is the orthogonal projection of $Q^{\prime}$ on $\hat{\mathcal{P}}$ with respect to the $L^{2}$-inner product induced by $\theta$, while $\bQ_{\tilde{\theta}}$ is the orthogonal projection of $Q^{\prime}_{\tilde{\theta}}$ with respect to the $L^{2}$-inner product induced by $\tilde{\theta}$; in particular $\phi\in \hat{\mathcal{P}}_{\theta}$ if and only if $\phi \in \hat{\mathcal{P}}_{\tilde{\theta}}$ and $\psi \in \mathcal{P}^{\perp}_{\theta}$ if and only if $e^{-2u}\psi \in \mathcal{P}^{\perp}_{\tilde{\theta}}$. Therefore, by denoting $\Gamma_{u}$ the orthogonal projection induced by $\tilde{\theta}$, one has $\Gamma_{u}(Q^{\prime}_{\tilde{\theta}})=f$. Let us also recall that in \cite{CaYa}, the authors show that the non-negativity of the Paneitz operator $P_{\theta}$ and the positivity of the CR-Yamabe invariant imply that $\bP$ is non-negative and $\ker \bP=\R$. Moreover, $\int_{M}Q^{\prime} =\int_{M}\bQ \leq 16\pi^{2}$ with equality if an only if $(M,T^{1,0}M,\theta)$ is the standard sphere; in particular the previously assumptions imply that the $(M,T^{1,0}M,\theta)$ is embeddable (see \cite{Chan}). Notice that unlike the Riemannian case, it remains unclear if the non-negativity of $\overline{P}'$ and $\ker \overline{P}'=\R$ is a sufficient condition for $\int_{M}Q'_{\theta}\leq 16\pi^{2}$. In particular, the results presented in this paper do not fully cover the case $\overline{P}'\geq 0$ and $\ker \overline{P}'=\R$.\\
Thus, from now on we will always assume that $(M,T^{1,0}M,\theta)$ is a pseudo-Einstein CR three dimensional manifold such that $\bP$ is non-negative and $\ker \bP=\R$. The problem in (\ref{eq2}) was first studied in \cite{CaYa} for $f$ constant and in the subcritical case, namely $\int_{M} \bQ  <16\pi^{2}$. Then in \cite{M1} the problem was solved for $f>0$ via a probabilistic approach again in the subcritical case; also, a solution of the problem was provided in \cite{QQ} for $f>0$ and $0<\int_{M}\bQ <16\pi^{2}$ via direct minimization.\\
In this paper we will study the equation (\ref{eq2}) allowing $f$ to change sign: our approach follows closely the methods in \cite{BFR}, where the authors study the analogous problem in the Riemannian setting. In particular, we will use a variational approach by defining a suitable functional on an appropriate space and then we will study the evolution problem along the negative gradient flow lines: the convergence at infinity will provide a solution to the initial problem. Indeed, with respect to \cite{BFR}, new technical issues will appear, which are essentially due to our sub-Riemannian setting: in particular all the computations and the estimates regarding the convergence along the flow lines have to be done accordingly to the projection on the space of $L^2$ CR pluriharmonic functions, that we defined earlier. Moreover, some technical estimates on the sphere will be adapted to the CR setting as we will see in Section 5 and the Appendix. Therefore, let us define the following functional $E:H\to \R$, by
$$E(u)=\int_{M}u \bP u + 2 \int_{M} \bQ u $$
where $H=\hat{\P}\cap \S^{2}(M)$ and $\S^{2}(M)$ is the Folland-Stein Sobolev space equipped with the equivalent norm (see section 2), defined by
\begin{equation}\label{norm}
\|u\|^{2}=\int_{M}u \bP u  + \int_{M} u^{2} .
\end{equation}
We consider the following space, which will serve as a constraint
$$X=\left\{u\in H; N(u):=\int_{M} \Gamma\left(f e^{2u}\right) = \int_{M}\bQ  \; \right\} ;$$
we notice that the space is well defined since $e^u \in L^2$, see \cite{CaYa2}, Theorem 3.1.

\noindent
As in the classical case, we will need the following hypotheses, depending on the sign of $\int_{M} \bQ$, namely:
\begin{equation}\label{hypotheses}
\left\{\begin{array}{lll}
(i)& \displaystyle\inf_{x\in M}f(x)<0, & \text{ if } \displaystyle\int_{M}\bQ <0\\
\\
(ii)& \displaystyle\sup_{x\in M}f(x)>0, \; \inf_{x\in M}f(x)<0  & \text{ if } \displaystyle\int_{M}\bQ =0\\
\\
(iii)& \displaystyle\sup_{x\in M}f(x)>0, & \text{ if } 0< \displaystyle\int_{M}\bQ \leq 16\pi^{2}.
\end{array}
\right.
\end{equation}

\noindent

In the case when $\int_{M}\overline{Q}'=0$, we let $\ell$ the unique CR pluriharmonic function satisfying $\overline{P}'_{\theta}\ell+\overline{Q}'=0$ and $\int_{M}\ell =0$, see \cite{CaYa2}, Theorem 1.1. Notice that $\overline{Q}'_{e^{\ell}\theta}=0$. We also recall that in the critical case $M=S^{3}$, there are some extra compatibility condition of Kazdan-Warner type that $f$ needs to satisfy in order to be the $\bQ$-curvature of a contact structure conformal to the standard one on the sphere (see Theorem 1.3. in \cite{QQ}).\\
Now, in order to define the flow equation, we compute the first variation of $E$, $N$, and their ($\S^2$) gradient, respectively:
\begin{align*}
\langle \nabla E(u),\phi \rangle & = 2 \int_{M} \left( \bP u + \bQ \right)\phi  \; , \forall \phi  \in H \; ,\\
\langle \nabla N(u),\phi \rangle & = 2 \int_{M} \Gamma\left( f e^{2u}\right)\phi  \; , \forall \phi  \in H \; ,\\
\nabla E(u) & =2\left(\bP +I \right)^{-1}\left(\bP u+\bQ \right) \; ,\\
\nabla N(u) & =2\left(\bP +I \right)^{-1}\Gamma\left( fe^{2u}\right) \; .
\end{align*}
In addition, since by hypotheses (\ref{hypotheses}), $\nabla N\neq 0$ on $X$, then $X$ is a regular hypersurface in $H$ and a unit normal vector field on $X$ is given by $\nabla N/\|\nabla N\|$. Indeed, $\nabla N(u)\not=0$ if and only if $\Gamma(e^{2u}f)\not=0$. This last identity is clear for the hypothesis $(i)$ and $(iii)$. But for $(ii)$, recall that $f\in \hat{\P}$, so if $\Gamma(e^{2u}f)=0$, then $\int_{M}e^{2u}f^{2}=0$, leading to a contradiction. The gradient of $E$ restricted to $X$ is then
$$\nabla^{X}E=\nabla E -\left\langle \nabla E, \frac {\nabla N}{\|\nabla N\|} \right\rangle\frac {\nabla N}{\|\nabla N\|}.$$
Finally, the (negative) gradient flow equation is given by
\begin{equation}\label{flowequation}
\left\{\begin{array}{lll}
\partial_t u = - \nabla^{X} E(u)\\
\\
u(0) = u_0 \in X
\end{array}
\right.
\end{equation}

\noindent
Now we can state our main results.
\begin{theorem}\label{negativecase}
Let $(M,T^{1,0}M,\theta)$ be a pseudo-Einstein CR three dimensional manifold such that $\bP$ is non-negative and $\ker \bP=\R$. Let us assume that $\displaystyle\int_{M}\bQ <0$ and let $f\in C(M)\cap \hat{\P}$ as in (\ref{hypotheses}). Then there exists a positive constant $C_{0}$ depending on $f^{-}=\max\{-f,0\}$, $M$ and $\theta$, such that if
$$e^{\tau \|u_{0}\|^{2}}\sup_{x\in M}f(x)\leq C_{0}$$
for a constant $\tau>1$ depending on $M$ and $\theta$, then as $t\to\infty$, the flow converges in $H$ to a solution $u_{\infty}$ of $(\ref{eq1})$. Moreover, there exist constants $B,\beta>0$ such that
$$\|u(t)-u_{\infty}\|\leq B(1+t)^{-\beta},$$
for all $t\geq 0$.
\end{theorem}

\begin{theorem}\label{zerocase}
Let $(M,T^{1,0}M,\theta)$ be a pseudo-Einstein CR three dimensional manifold such that $\bP$ is non-negative and $\ker \bP=\R$. Let us assume that $\displaystyle\int_{M}\bQ =0$ and let $f\in C(M)\cap \hat{\P}$ as in (\ref{hypotheses}). Then as $t\to\infty$, the flow converges in $H$ to a function $u_{\infty}$ and there exists a constant $\lambda$ such that $v=u_{\infty}+\lambda$ satisfies
$$\bP v+\bQ=\delta \Gamma(fe^{2v}) ,$$
where $\delta\in \{+1,0,-1\}$. Moreover, there exist constants $B,\beta>0$ such that
$$\|u(t)-u_{\infty}\|\leq B(1+t)^{-\beta},$$
for all $t\geq 0$. If in addition, we assume that $\int_{M} fe^{2\ell}\not=0$, then $\delta\not=0$.
\end{theorem}

\begin{theorem}\label{positivecase}
Let $(M,T^{1,0}M,\theta)$ be a pseudo-Einstein CR three dimensional manifold such that $\bP$ is non-negative and $\ker \bP=\R$. Let us assume that $0< \displaystyle\int_{M}\bQ <16\pi^{2}$ and let $f\in C(M)\cap \hat{\P}$ as in (\ref{hypotheses}). Then as $t\to\infty$, the flow converges in $H$ to a solution $u_{\infty}$ of $(\ref{eq1})$. Moreover, there exist constants $B,\beta>0$ such that
$$\|u(t)-u_{\infty}\|\leq B(1+t)^{-\beta},$$
for all $t\geq 0$.
\end{theorem}

\noindent
Finally, the critical case of the sphere, which is a bit different. We will consider a group $G$ acting on $S^{3}$ preserving the CR structure. We denote by $\Sigma$ the set of points fixed by $G$, that is
$$\Sigma=\left\{x\in S^{3}; \; g\cdot x=x, \;  \forall g\in G \right\}$$
and we will assume $f$ being invariant under $G$, namely $f(g\cdot x)=f(x), \forall g\in G$. Then we have the following
\begin{theorem}\label{criticalcase}
Let us consider the sphere $M=S^3$ equipped with its standard contact structure and let $f\in C(M)\cap \hat{\P}$ as in (\ref{hypotheses}) and invariant under $G$. Let us assume that also $u_0 \in X$ is invariant under $G$. If $\Sigma=\emptyset$ or
$$\sup_{x\in \Sigma}f(x)\leq e^{-\frac{E(u_{0})}{16\pi^{2}}} , $$
then as $t\to\infty$, the flow converges in $H$ to a solution (invariant under $G$) $u_{\infty}$ of $(\ref{eq1})$. Moreover, there exist constants $B,\beta>0$ such that
$$\|u(t)-u_{\infty}\|\leq B(1+t)^{-\beta},$$
for all $t\geq 0$.
\end{theorem}

\vspace{5mm}

{\noindent\bf Acknowledgment}

\noindent
The authors wants to express their gratitude to the referee for his/her careful reading of the paper: the remarks and the suggestions led to a serious improvement of the paper.


\section{Some definition in pseudo-Hermitian geometry}

\noindent
We will follow the notations in \cite{CaYa}. Let $M^3$ be a smooth, oriented three-dimensional manifold.  A CR structure on $M$ is a one-dimensional complex sub-bundle $T^{1,0}M\subset T_{\C}M:= TM\otimes\C$ such that $T^{1,0}M\cap T^{0,1}M=\{0\}$ for $T^{0,1}M:=\overline{T^{1,0}M}$.  Let $\mathcal{H}=Re T^{1,0}M$ and let $J\colon \mathcal{H}\to \mathcal{H}$ be the almost complex structure defined by $J(Z+\bar Z)=i(Z-\bar Z)$, for all $Z\in T^{1,0}M$. The condition that $T^{1,0}M\cap T^{0,1}M=\{0\}$ is equivalent to the existence of a contact form $\theta$ such that $\ker \theta =\mathcal{H}$. We recall that a 1-form $\theta$ is said to be a contact form if $\theta\wedge d\theta$ is a volume form on $M^{3}$. Since $M$ is oriented, a contact form always exists, and is determined up to multiplication by a positive real-valued smooth function. We say that $(M^3,T^{1,0}M)$ is strictly pseudo-convex if the Levi form $d\theta(\cdot,J\cdot)$ on $\mathcal{H}\otimes \mathcal{H}$ is positive definite for some, and hence any, choice of contact form $\theta$. We shall always assume that our CR manifolds are strictly pseudo-convex.\\
Notice that in a CR-manifold, there is no canonical choice of the contact form $\theta$. A pseudo-Hermitian manifold is a triple $(M^3,T^{1,0}M,\theta)$ consisting of a CR manifold and a contact form. The Reeb vector field $T$ is the vector field such that $\theta(T)=1$ and $d\theta(T,\cdot)=0$. The choice of $\theta$ induces a natural $L^{2}$-dot product $\langle\cdot,\cdot\rangle$, defined by
$$\langle f,g\rangle =\int_{M}f(x)g(x)\ \theta\wedge d\theta.$$
A $(1,0)$-form is a section of $T_{\C}^\ast M$ which annihilates $T^{0,1}M$.  An admissible coframe is a non-vanishing $(1,0)$-form $\theta^1$ in an open set $U\subset M$ such that $\theta^1(T)=0$.  Let $\theta^{\bar 1}:=\overline{\theta^1}$ be its conjugate.  Then $d\theta=ih_{1\bar 1}\theta^1\wedge\theta^{\bar 1}$ for some positive function $h_{1\bar 1}$.  The function $h_{1\bar 1}$ is equivalent to the Levi form.  We set $\{Z_1,Z_{\bar 1},T\}$ to the dual of $(\theta^{1},\theta^{\bar 1},\theta)$. The geometric structure of a CR manifold is determined by the  connection form $\omega_1{}^1$ and the torsion form $\tau_1=A_{11}\theta^1$ defined in an admissible coframe $\theta^1$ and is uniquely determined by
\begin{equation*}
\left\{\begin{array}{ll}
d\theta^1 = \theta^1\wedge\omega_1{}^1 + \theta\wedge\tau^1, \\
\omega_{1\bar 1} + \omega_{\bar 11} = dh_{1\bar 1},
\end{array}
\right.
\end{equation*}
where we use $h_{1\bar 1}$ to raise and lower indices.  The connection forms determine the pseudo-Hermitian connection $\nabla$, also called the Tanaka-Webster connection, by
\[ \nabla Z_1 := \omega_1{}^1\otimes Z_1. \]
The scalar curvature $R$ of $\theta$, also called the Webster curvature, is given by the expression
 \[ d\omega_1{}^1 = R\theta^1\wedge\theta^{\bar 1} \mod\theta . \]
\begin{definition}
A real-valued function $w\in C^\infty(M)$ is CR pluriharmonic if locally $w=Re f$ for some complex-valued function $f\in C^\infty(M,\C)$ satisfying $Z_{\bar 1}f=0$.
\end{definition}

\noindent
Equivalently, \cite{Lee}, $w$ is a CR pluriharmonic function if
\[ P_{3} w:=\nabla_1\nabla_1\nabla^1 w + iA_{11}\nabla^1 w = 0 \]
for $\nabla_1:=\nabla_{Z_1}$. We denote by $\mathcal{P}$ the space of all CR pluriharmonic functions and $\hat{\P}$ the completion of $\P$ in $L^{2}(M)$, also called the space of $L^{2}$ CR pluriharmonic functions. Let $\Gamma: L^{2}(M)\to \hat{\mathcal{P}}$ be the orthogonal projection on the space of $L^{2}$ pluriharmonic functions. If $S:L^{2}(M)\to \ker \bar\partial_{b}$ denotes the Szego kernel, then
\begin{equation}\label{szd}
\Gamma=S+\bar S+F,
\end{equation}
where $F$ is a smoothing kernel as shown in \cite{H}.
The Paneitz operator $P_\theta$ is the differential operator
\begin{align*}
P_\theta(w) & := 4\text{div}(P_{3}w) \\
& = \Delta_b^2w + T^2 - 4\text{Im} \nabla^1\left(A_{11}\nabla^1f\right)
\end{align*}
for $\Delta_b:=\nabla^1\nabla_1+\nabla^{\bar 1}\nabla_{\bar 1}$ the sub-Laplacian.  In particular, $\mathcal{P}\subset\ker P_\theta$. Hence, $\ker P_{\theta}$ is infinite dimensional. For a thorough study of the analytical properties of $P_{\theta}$ and its kernel, we refer the reader to \cite{H,CCY,CaYa1}. The main property of the Paneitz operator $P_{\theta}$ is that it is CR covariant \cite{Hir2}. That is, if $\hat\theta=e^w\theta$, then $e^{2w} P_{\hat{\theta}}=P_\theta$.
\begin{definition}\label{pprimep}
Let $(M^3,T^{1,0}M,\theta)$ be a pseudo-Hermitian manifold. The Paneitz type operator $P_{\theta}^\prime\colon\mathcal{P}\to C^\infty(M)$ is defined by
\begin{align}
P_\theta^\prime f & = 4\Delta_b^2 f - 8 \textnormal{Im}\left(\nabla^\alpha(A_{\alpha\beta}\nabla^\beta f)\right) - 4 \textnormal{Re}\left(\nabla^\alpha(R\nabla_\alpha f)\right) \notag\\
& \quad + \frac{8}{3}\textnormal{Re} (\nabla_\alpha R - i\nabla^\beta A_{\alpha\beta})\nabla^\alpha f - \frac{4}{3}f\nabla^\alpha( \nabla_\alpha R - i\nabla^\beta A_{\alpha\beta})\label{pprime1}
\end{align}
for $f\in\mathcal{P}$.
\end{definition}

\noindent
The main property of the operator $P_{\theta}^\prime$ is its ``almost" conformal covariance as shown in \cite{BG,CaYa}. That is if  $(M^3,T^{1,0}M,\theta)$ is a pseudo-Hermitian manifold, $w\in C^\infty(M)$, and we set $\hat\theta=e^w\theta$, then
\begin{equation} \label{pprime}
e^{2w}P_{\hat{\theta}}^\prime(u) = P_\theta^\prime(u) + P_\theta\left(uw\right)
\end{equation}
for all $u\in\mathcal{P}$.  In particular, since $P_\theta$ is self-adjoint and $\mathcal{P}\subset \ker P_{\theta}$, we have that the operator $P^\prime$ is conformally covariant, mod $\mathcal{P}^\perp$.
\begin{definition}
A pseudo-Hermitian manifold $(M^3,T^{1,0}M,\theta)$ is pseudo-Einstein if
$$\nabla_\alpha R - i\nabla^\beta A_{\alpha\beta}=0 .$$
\end{definition}

\noindent
Moreover, if $\theta$ induces a pseudo-Einstein structure then $e^{u}\theta$ is pseudo-Einstein if and only if $u\in \mathcal{P}$. The definition above was stated in \cite{CaYa}, but it was implicitly mentioned in \cite{Hir2}. In particular, if $(M^3,T^{1,0}M,\theta)$ is pseudo-Einstein, then $P_{\theta}^\prime$ takes a simpler form:
$$P_\theta^\prime f = 4\Delta_b^2 f - 8 \text{Im}\left(\nabla^1(A_{11}\nabla^1 f)\right) - 4\text{Re}\left(\nabla^1(R\nabla_1 f)\right).$$
In particular, one has
$$\int_{M}uP'_{\theta}u \geq 4\int_{M}|\Delta_{b}u|^{2} -C\int_{M}|\nabla_{b}u|^{2}.$$
Using the interpolation inequality $$\int_{M}|\nabla_{b}u|^{2}\leq C\|u\|_{L^{2}}\|\Delta_{b}u\|_{L^{2}},$$
and $2ab\leq \varepsilon a^{2}+\frac{1}{\epsilon} b^{2}$, we have the existence of $C_{1}>0$ and $C_{2}>0$, such that
$$ \int_{M}uP'_{\theta}u \geq C_{1}\int_{M}|\Delta_{b}u|^{2}-C_{2}\int_{M}u^{2}.$$
Hence, if $P'_{\theta}$ is non-negative, with trivial kernel, one has the equivalence of the Folland-Stein Sobolev norm and $(\ref{norm})$.
 
\begin{definition}\label{qprime}
Let $(M^3,T^{1,0}M,\theta)$ be a pseudo-Einstein manifold. The $Q^\prime$-curvature is the scalar quantity defined by
\begin{equation}\label{qprime1}
Q_\theta^\prime = 2\Delta_b R - 4 |A|^2 + R^2.
\end{equation}
\end{definition}

\noindent
The main equation that we will be dealing with is the change of the $Q^{\prime}$-curvature under conformal change. Let $(M^3,T^{1,0}M,\theta)$ be a pseudo-Einstein manifold, let $w\in\mathcal{P}$, and set $\hat\theta=e^w\theta$. Hence $\hat\theta$ is pseudo-Einstein. Then \cite{BG,CaYa}
\begin{equation}
 e^{2w} Q_{\hat{\theta}}^\prime = Q_\theta^\prime + P_\theta^\prime(w) + \frac{1}{2}P_\theta\left(w^2\right) .
\end{equation}
In particular, $Q_\theta^\prime$ behaves as the $Q$-curvature for $P_\theta^\prime$, mod $\mathcal{P}^\perp$. Since we are working modulo $\mathcal{P}^{\perp}$ it is convenient to project the previously defined quantities on $\hat{\mathcal{P}}$. So we define the operator $\overline{P}_{\theta}^{\prime}=\Gamma \circ P_{\theta}^{\prime}$ and the $\overline{Q}^{\prime}$-curvature by $\overline{Q}^{\prime}_{\theta}=\Gamma( Q^{\prime}_{\theta})$. Notice that
$$\int_{M}Q^{\prime}_{\theta}\ \theta\wedge d\theta =\int_{M}\overline{Q}^{\prime}_{\theta}\ \theta\wedge d\theta.$$
Moreover, the operator $\bP_{\theta}$ has many interesting analytical properties. Indeed, $\bP_{\theta}:\mathcal{P}\to \hat{\mathcal{P}}$ is an elliptic pseudo-differential operator (see \cite{CaYa2}) and if we assume that $\ker \bP_{\theta}=\R$, then its Green's function $G$ satisfies
$$\bP_{\theta}G(\cdot,y)=\Gamma(\cdot,y)-\frac{1}{V},$$
where $V=\int_{M}\theta\wedge d\theta$ is the volume of $M$ and $\Gamma(\cdot,\cdot)$ is the kernel of the projection operator $\Gamma$. Moreover,
$$G(x,y)=-\frac{1}{4\pi^{2}}\ln(|xy^{-1}|)+\mathcal{K}(x,y),$$
where $\mathcal{K}$ is a bounded kernel as proved in \cite{CCY2}.


\section{Preliminary results on the flow}

\noindent
First we recall one fundamental inequality that we will be using all along this paper, namely the CR version of the Beckner-Onofri inequality. This inequality was first proved in the odd dimensional spheres in \cite{Bran} and then naturally extended to pseudo-Einstein 3-manifolds in \cite[Theorem~3.1]{CaYa2}.

\begin{theorem}\label{thmbran}
Assume that $\overline{P}'$ is non-negative and $\ker\overline{P}'=\R$. Then, there exists $C>0$  such that for all $u\in \hat{\mathcal{P}}\cap \mathcal{S}^{2}(M)$ with $\displaystyle\int_{M}u=0$, we have
$$\frac{1}{16\pi^{2}}\int_{M}u\bP u+C\geq \ln\left(\dashint_{M}e^{2u}\right).$$
\end{theorem}

\noindent
In the case of the sphere, $C$ can be taken to be $0$ and equality holds if and only if $u=J(h)$ with $h\in Aut(S^{3})$ and $J(h)=det(Jac(h))$ is the determinant of the Jacobian determinant of $h$. The dual version of the above inequality was also investigated in \cite{M2}, where the existence of extremals was investigated.\\
Now, we prove the global existence of solutions of (\ref{flowequation}):
\begin{lemma}\label{globalexistence}
Let $(M,T^{1,0}M,\theta)$ be a pseudo-Einstein CR three dimensional manifold such that $\bP$ is non-negative and $\ker \bP=\R$. Let $f\in C(M)\cap \hat{\P}$ as in (\ref{hypotheses}). Then for any $u_0\in X$ there exists a solution $u\in C^{\infty}([0,\infty),H)$ of problem (\ref{flowequation}) such that $u(t) \in X$, for all $t\geq 0$. Moreover it holds
$$\int_{0}^{t}\|\partial_{s}u(s)\|^{2}ds=E(u_{0})-E(u(t)),$$
for all $t\geq 0$.
\end{lemma}
\begin{proof}
Since all the functionals involved are regular, the short time existence of a solution $u$ for (\ref{flowequation}) is ensured by the Cauchy-Lipschitz Theorem. In order to extend it to all $t\geq 0$, we notice that
$$\|\partial_{t}u\| = \|\nabla^{X} E(u)\| \leq 2\|\nabla E(u)\|\leq C_{1}\|u\|+C_{2}.$$
Thus, since
$$\partial_{t}\|u\|^{2}=2\langle u, \partial_{t}u\rangle \leq C_{3}\|u\|^{2}+C_{4},$$
by Gronwall's lemma, the solution $u$ exists for all $t\geq 0$. In addition
$$\partial_t N(u)=\left\langle \nabla N (u), \partial_t u\right\rangle=-\left\langle \nabla N (u), \nabla^{X} E(u)\right\rangle=0, $$
therefore $u(t)\in X$, for all $t\geq 0$. Finally, we have
$$\partial_{t}E(u)=\langle \nabla E(u),\partial_{t}u\rangle =-\|\partial_{t}u\|^{2}.$$
Hence, $E$ is decreasing along the flow and the following energy identity holds
\begin{equation}\label{energy identity}
\int_{0}^{t}\|\partial_{s}u\|^{2}ds=E(u_{0})-E(u(t)).
\end{equation}
\end{proof}

\noindent
Next we prove the following lemma about the convergence
\begin{lemma}\label{convergence}
Let $(M,T^{1,0}M,\theta)$ be a pseudo-Einstein CR three dimensional manifold such that $\bP$ is non-negative and $\ker \bP=\R$. Let $f\in C(M)\cap \hat{\P}$ as in (\ref{hypotheses}) and let $u$ be the solution of problem (\ref{flowequation}) obtained in the previous Lemma (\ref{globalexistence}). If there exists a constant $C>0$ such that $\|u(t)\|\leq C$, for all $t\geq 0$, then when $t\to \infty$, $u(t)\to u_{\infty}$ in $H$ and $u_{\infty}$ solves the equation
$$\overline{P}'u+\overline{Q}'u=\lambda\Gamma (fe^{2u}),$$
for a certain $\lambda \in \R$.  Moreover, there exist constants $B,\beta>0$ such that
$$\|u(t)-u_{\infty}\|\leq B(1+t)^{-\beta},$$
for all $t\geq 0$.
\end{lemma}

\begin{proof}
Since $\|u\|\leq C$, we have that
$$|E(u)|\leq 2\|u\|^{2}+C_{2}.$$
Therefore, by the previous energy estimate
$$\int_{0}^{\infty}\|\partial_{t}u\|^{2}dt <\infty.$$
So there exists a sequence $t_{k}\to \infty$ such that
$$\|\partial_{t}u(t_{k})\|=\| \nabla^{X} E(u(t_k))\|\to 0.$$
Now, from the boundedness of $\|u\|$, we also have the convergence $u(t_{k})\to u_{\infty}$ strongly in $L^{2}(M)$ and weakly in $\S^{2}(M)$. From Theorem \ref{thmbran}, we have that $e^{2u(t_{k})}\in L^{p}(M)$, with $p\geq 1$, and $\|e^{2u(t_{k})}\|_{L^{p}}$ is uniformly bounded. Thus by Egorov's lemma, we can deduce that
$$\left\|fe^{2u(t_{k})}-fe^{2u_{\infty}}\right\|_{L^{p}}\to 0, \; 1\leq p<\infty.$$
Indeed, we fix $\varepsilon>0$. Then, there exists a set $A$ with $Vol(A)<\varepsilon$ such that $fe^{2u(t_{k})}$ converges uniformly to $fe^{2u_{\infty}}$ on $M\setminus A$. Therefore,
$$\left\|fe^{2u(t_{k})}-fe^{2u_{\infty}}\right\|_{L^{p}}\leq  C\left\|fe^{2u(t_{k})}-fe^{2u_{\infty}}\right\|_{L^{\infty}(M\setminus A)}+\Big(\left\|fe^{2u(t_{k})}\right\|_{L^{\tilde{p}}}+\left\|fe^{2u_{\infty}}\right\|_{L^{\tilde{p}}}\Big)Vol(A)^{\frac{1}{p}-\frac{1}{\tilde{p}}},$$
for p<$\tilde{p}<\infty$. So the conclusion follows from the uniform boundedness of $\|e^{2u(t_{k})}\|_{L^{p}}$, for all $1\leq p<\infty$.\\

Thus $u_{\infty}\in X$. Now we have that
$$\nabla N (u(t_k))=2(\overline{P}'+I)^{-1}\Gamma \left(fe^{2u(t_k)}\right),$$
and since $fe^{2u(t_{k})}$ converges strongly in $L^{p}(M)$ and $\Gamma$ maps continuously $L^{p}(M)$ to $L^{p}(M)$ (this follows from (\ref{szd}) and \cite{PS}), we have by the compactness of $\overline{P}'+I$ that $\nabla N(u(t_{k})$ converges strongly to $\nabla N(u_{\infty})$. Also,
$$\nabla E(u(t))=2(\overline{P}'+I)^{-1}(\overline{P}'u+\overline{Q}')=2u(t)+2(\overline{P}'+I)^{-1}(\overline{Q}'-u).$$
Thus, since $\partial_{t}u(t_{k})\to 0$ in $\S^{2}(M)$, we have that $u(t_{k})$ converges strongly to $u_{\infty}$ and $ \nabla^{X} E(u_{\infty})=0$. Moreover, we have
$$(\overline{P}'+I)^{-1}(\overline{P}'u_{\infty}+\overline{Q}')=\lambda(u_{\infty})(\overline{P}'+I)^{-1}\Gamma(fe^{2u_{\infty}})$$
where
$$\lambda (u_{\infty})=\frac{\langle \nabla E(u_{\infty}),\nabla N(u_{\infty})\rangle}{\|\nabla N(u_{\infty})\|^{2}} .$$
Now by integration we have that
$$\int_{M}\overline{Q}'=\lambda(u_{\infty})\int_{M} \Gamma \left(fe^{2u_{\infty}}\right) ,$$
and since $u_{\infty}\in X$, we have if $\displaystyle\int_M \overline{Q}' \not =0$, that $\lambda(u_{\infty})=1$ and hence $u_{\infty}$ solves the desired equation. On the other hand, if $\displaystyle\int_{M}\overline{Q}'=0$, either $\lambda(u_{\infty})=0$ and thus
$$\overline{P}'u_{\infty}+\overline{Q}'=0,$$
or $\lambda(u_{\infty})>0$, thus setting $v=u_{\infty}+\frac{1}{2}\ln(\lambda(u_{\infty}))$ we have
$$\overline{P}'v+\overline{Q}'=\Gamma(fe^{2v}),$$
and similarly if $\lambda(u_{\infty})<0$, we have a function $v=u_{\infty}+\frac{1}{2}\ln(-\lambda(u_{\infty}))$ such that
$$\overline{P}'v+\overline{Q}'=-\Gamma(fe^{2v}).$$
In particular, if we assume that $\int_{M}fe^{2\ell}\not=0$ in the case $\lambda(u_{\infty})=0$, we have $u_{\infty}-\ell$ is constant. Hence, $$0=\int_{M}fe^{2u_{\infty}}=e^{2(u_{\infty}-\ell)}\int_{M}fe^{2\ell}\not=0,$$ which is a contradiction.\\
The polynomial convergence of the flow can be deduced from the Lojasiewicz-Simon inequality following Theorem 3 in \cite{simon} and Lemma 3.2 in \cite{BFR}. Let $\eta:H\to T_{u_{\infty}}X$ be the natural projection, where $T_{u_{\infty}}X$ denotes the tangent space of the manifold $X$ at the point $u_\infty$. We have, for $v\in T_{u_{\infty}}X$
$$\left(\nabla^{X}\right)^{2}E(u_{\infty})v=\eta \left(\nabla^{2}E(u_{\infty})v-\frac{\langle \nabla E(u_{\infty}),\nabla N(u_{\infty})\rangle}{\|\nabla N(u)\|^{2}}\nabla^{2}N(u_{\infty})v+R^{\perp}v\right)$$
where $R^{\perp}v$ is the component along $\nabla N(u_{\infty})$. Thus, since $\eta \left(R^{\perp}v\right)=0$, we have that
$$\left(\nabla^{X}\right)^{2}E(u_{\infty})v=2\left(I-\eta(\overline{P}'+1)\right)v-4\frac{\langle \nabla E(u_{\infty}),\nabla N(u_{\infty})\rangle}{\|\nabla N(u)\|^{2}}(\overline{P}'+1)^{-1}\Gamma (fe^{2u_{\infty}}v).$$
It can be checked that  $(\nabla^{X})^{2} E(u_{\infty}):T_{u_{\infty}}X\to T_{u_{\infty}}X$ is a Fredholm operator, then there exists a constant $\delta>0$ and $0<\kappa<\frac{1}{2}$ such that if $\|u(t)-u_{\infty}\|<\delta$, it holds
$$\|\nabla ^{X}E(u)\|\geq (E(u(t))-E(u_{\infty}))^{1-\kappa}.$$
We note that if $E(u(t_{0}))=E(u_{\infty})$ for some $t_0\geq0$, then the flow is stationary and the estimate is trivially satisfied. So we can assume that $E(u(t))-E(u_{\infty})>0$, for all $t\geq 0$. Since $\lim_{n\to \infty}\|u(t_{n})-u_{\infty}\|=0$, for a given $\varepsilon>0$, there exists $n_{0}>0$  such that for $n\geq n_{0}$ we have,
$$\|u(t_{n})-u_{\infty}\|<\frac{\varepsilon}{2}$$
and
$$\frac{1}{\kappa}(E(u(t_{n}))-E(u_{\infty}))^{\kappa}<\frac{\varepsilon}{2}.$$
We set $\varepsilon=\frac{\delta}{2}$ and
$$T:=\sup\left\{t\geq t_{n_{0}}; \|u(s)-u_{\infty}\|<\delta ;s\in[t_{n_{0}},t]\right\},$$
and we assume for the sake of contradiction that $T<\infty$. Now we have
$$-\partial_{t}[E(u(t))-E(u_{\infty})]^{\kappa}=-\kappa \partial_{t}E(u(t))[E(u(t))-E(u_{\infty})]^{\kappa-1},$$
but
$$-\partial_{t}E(u(t))=-\langle E(u),\partial_{t}u\rangle =\|\nabla^{X}E(u)\|\|\partial_{t}u\|.$$
Thus, for $t\in [t_{n_{0}},T]$ we have
$$-\partial_{t}[E(u(t))-E(u_{\infty})]^{\kappa}\geq \kappa \|\partial_{t}u\|,$$
and since $E$ is non-increasing along the flow, we have after integration in the interval $[t_{n_{0}},T]$
$$\|u(T)-u(t_{n_{0}})\|\leq \int_{t_{n_{0}}}^{T}\|\partial_{s}u\|ds\leq\frac{1}{\kappa}[E(u(t_{n_{0}}))-E(u_{\infty})]^{\kappa}<\frac{\varepsilon}{2} .$$
Hence,
$$\|u(T)-u_{\infty}\|\leq \|u(T)-u(t_{n_{0}})\|+\|u(t_{n_{0}})-u_{\infty}\|<\varepsilon=\frac{\delta}{2}$$
which is a contradiction and so $T=+\infty$. We set now $g(t)=E(u(t))-E(u_{\infty})$, for $t\in [t_{n_{0}},+\infty)$. Then we have
$$g'(t)=-\|\nabla^{X}E(u)\|^{2}\geq g^{2\kappa-1}(t).$$
By integration we have
$$g^{2\kappa-1}(t)\geq g^{2\kappa-1}(t_{n_{0}})+(1-2\kappa)(t-t_{n_{0}}).$$
Since $2\kappa-1<0$, we have
$$g(t)\leq [g^{2\kappa-1}(t_{n_{0}})+(1-2\kappa)(t-t_{n_{0}})]^{\frac{1}{2\kappa-1}}\leq Ct^{\frac{1}{2\kappa-1}} .$$
Now, by taking $t'>t$, we have
$$\|u(t)-u(t')\|\leq \int_{t}^{t'}\|\partial_{s}u\|ds\leq \frac{1}{\theta}[E(u(t))-E(u_{\infty})]^{\kappa}\leq \frac{1}{\kappa}g^{\kappa}(t)\leq Ct^{\frac{\kappa}{2\kappa-1}}.$$
For $t'=t_{n}$, letting $n\to \infty$ and setting $\beta=\frac{\kappa}{1-2\kappa}$, we get that for $t>t_{n_{0}}$
$$\|u(t)-u_{\infty}\|\leq Ct^{-\beta}$$
Therefore, since $\|u(t)-u_{\infty}\|$ is bounded for $t>t_{n_{0}}$, we have the existence of $B>0$ such that for all $t\geq 0$
$$\|u(t)-u_{\infty}\|\leq B(1+t)^{-\beta}.$$
\end{proof}

\begin{corollary}\label{average}
Let $(M,T^{1,0}M,\theta)$ be a pseudo-Einstein CR three dimensional manifold such that $\bP$ is non-negative and $\ker \bP=\R$. Let $f\in C(M)\cap \hat{\P}$ as in (\ref{hypotheses}) and let $u$ be the solution of problem (\ref{flowequation}) obtained in the Lemma (\ref{globalexistence}). If $\bar u=\displaystyle\frac{1}{V}\displaystyle\int_{M} u$ is uniformly bounded then the flow converges. Here $V=\displaystyle\int_M \theta\wedge d\theta$ is the volume of $M$.
\end{corollary}
\begin{proof}
From the energy identity (\ref{energy identity}) we have that
$$\int_{M}u\overline{P}'u+2\int_{M}\overline{Q}'u \leq E(u_{0}),$$
but we also have from the Poincar\'{e}-type inequality (or the non-negativity of the operator $\overline{P}'$), that
$$\int_{M}u\overline{P}'u\geq \lambda_{1}\int_{M}(u-\bar{u})^{2}.$$
Here $\lambda_{1}$ is the first non-zero eigenvalue of the operator $\overline{P}'$. In particular, from Young's inequality, we have that
$$\int_{M}u\overline{P}'u\leq E(u_{0})+\varepsilon\int_{M}(u-\bar{u})^{2}+C(\varepsilon)\|\overline{Q}'\|_{L^{2}}^{2} - 2\bar u \int_{M}\overline{Q}'$$
Hence, for $\varepsilon$ small enough, we have that
$$\int_{M}u\overline{P}'u\leq C,$$
since $\bar{u}$ is uniformly bounded, then the uniform boundedness of $\|u\|$ and the conclusion follows from Lemma \ref{convergence}.
\end{proof}

\noindent
Therefore, in the rest of the paper,  we will show the uniform boundedness of $\bar{u}$ along the flow, in order to have convergence at infinity.


\section{The sub-critical case}

\noindent
Along all this section we will assume that $\bP$ is non-negative and $\ker \bP=\R$. Next we consider the three separate cases in which $\displaystyle\int_{M}\bQ <16\pi^{2}$. Also we let $V=\displaystyle\int_M \theta\wedge d\theta$ be the volume of $M$.


\subsection{Case $\displaystyle\int_{M}\bQ<0$ and proof of Theorem \ref{negativecase}}

\begin{lemma}\label{concentration1}
There exists a positive constant $C>0$ depending on $M$ and $\theta$ such that for any measurable subset $K\subset M$ with $Vol(K)>0$, we have
$$\int_{M}u\leq |E(u_{0})|+\frac{C}{Vol(K)}+\frac{4 V}{Vol(K)}\max\left(\int_{K}u,0\right)$$
\end{lemma}

\begin{proof}
Without loss of generality we can assume that $\displaystyle\int_{M}u>0$ otherwise the inequality is trivially satisfied. First we have that
$$\int u\overline{P}'u\leq E(u_{0})-2\int_{M}\overline{Q}'u$$
and
$$\|u-\bar{u}\|_{L^{2}}^{2}\leq \frac{1}{\lambda_{1}}\int_{M}u\overline{P}'u .$$
Hence,
$$\int_{M}u^{2}\leq \frac{1}{\lambda_{1}}E(u_{0})-\frac{2}{\lambda_{1}}\int_{M}\overline{Q}'u +\frac{1}{V}\left(\int_{M}u\right)^{2}$$
Now if $\displaystyle\int_{K}u \leq 0$, then we have
$$\left(\int_{M}u\right)^{2}\leq \left(\int_{K^{c}}u\right)^{2}\leq Vol(K^{c})\int_{M}u^{2},$$
hence
$$\frac{Vol(K)}{V}\int_{M}u^{2}\leq \frac{1}{\lambda_{1}}E(u_{0})-\frac{2}{\lambda_{1}}\int_{M}\overline{Q}'u.$$
Again using Young's inequality we have
$$\int_{M}u^{2}\leq \frac{2V}{\lambda_{1}Vol(K)}E(u_{0})+\frac{4\|\overline{Q}'\|_{L^{2}}^{2}V^{2}}{\lambda_{1}^{2}Vol(K)^{2}},$$
but
\begin{align}
\left(\int_{M}u\right)^{2}&\leq V\int_{M}u^{2} \leq  \frac{2V^{2}}{\lambda_{1}Vol(K)}E(u_{0})+\frac{4\|\overline{Q}'\|_{L^{2}}^{2}V^{3}}{\lambda_{1}^{2}Vol(K)^{2}}\notag\\
&\leq |E(u_{0})|^{2}+\frac{V^{4}}{\lambda_{1}^{2}Vol(K)^{2}}+\frac{4\|\overline{Q}'\|_{L^{2}}^{2}V^{3}}{\lambda_{1}^{2}Vol(K)^{2}}\notag,
\end{align}
which yields
$$\int_{M}u \leq |E(u_{0})|+\frac{C}{Vol{K}}.$$
We assume now that $\displaystyle\int_{K}u> 0$. Then one has
$$\int_{M}u^{2}\leq \frac{1}{\lambda_{1}}E(u_{0})-\frac{2}{\lambda_{1}}\int_{M}\overline{Q}'u +\frac{1}{V}\left(\left(\int_{K}u\right)^{2}+\left(\int_{K^{c}}u\right)^{2}+2\int_{K^{c}}u\int_{K}u \right),$$
and
$$\frac{2}{V}\int_{K^{c}}u\int_{K}u\leq \frac{2Vol(K^{c})}{Vol(K)V}\left(\int_{K}u\right)^{2}+\frac{Vol(K)}{2V}\int_{M}u^{2}.$$
Hence,
$$\frac{Vol(K)}{2V}\int_{M}u^{2}\leq \frac{1}{\lambda_{1}}E(u_{0})-\frac{2}{\lambda_{1}}\int_{M}\overline{Q}'u+\frac{3}{Vol(K)}\left(\int_{K}u\right)^{2}.$$
By using that
$$\left|\frac{2}{\lambda_{1}}\int_{M}\overline{Q}'u\right|\leq \frac{Vol(K)}{4V}\int_{M}u^{2} +\frac{4V\|\overline{Q}'\|_{L^{2}}^{2}}{\lambda_{1}^{2}Vol(K)},$$
we have,
$$\int_{M}u^{2}\leq \frac{4V}{\lambda_{1}Vol(K)}|E(u_{0})|+\frac{16V^{2}\|\overline{Q}'\|_{L^{2}}^{2}}{\lambda_{1}^{2}Vol(K)^{2}}+ \frac{12V}{Vol(K)^{2}}\left(\int_{M}u\right)^{2}.$$
Hence,
$$\left(\int_{M}u\right)^{2}\leq  \frac{4V^{2}}{\lambda_{1}Vol(K)}|E(u_{0})|+\frac{16V^{3}\|\overline{Q}'\|_{L^{2}}^{2}}{\lambda_{1}^{2}Vol(K)^{2}}+
\frac{12V^{2}}{Vol(K)^{2}}\left(\int_{M}u\right)^{2},$$
and therefore
$$\int_{M}u\leq |E(u_{0})|+\frac{C}{Vol(K)}+\frac{4V}{Vol(K)}\int_{K}u .$$
\end{proof}

\begin{lemma}\label{concentration2}
Let $K$ be a measurable subset of $M$ such that $Vol(K)>0$. Then there exists a constant $\alpha>1$ depending on $M$ and $\theta$ and a constant $C_{K}>1$ depending on $Vol(K)$ such that
$$\int_{M}e^{2u}\leq C_{K}e^{\alpha\|u_{0}\|^{2}}\max\left(\left(\int_{K}e^{2u}\right)^{\alpha},1\right).$$
\end{lemma}

\begin{proof}
Recall that from Theorem \ref{thmbran} 
one has the existence of $C>0$ such that
$$\int_{M}e^{2u}\leq C\exp\left(\frac{1}{16\pi^{2}}\int_{M}u\overline{P}'u+\frac{2}{V}\int_{M}u\right).$$
Again, by the energy identity (\ref{energy identity}) and Young's inequality, we have
\begin{align}
\int_{M}u\overline{P}'u & \leq E(u_{0})-2\int_{M}\overline{Q}'(u-\bar{u})-2\bar{u}\int_{M}\overline{Q}'\notag\\
&\leq E(u_{0})-2\bar{u}\int_{M}\overline{Q}'+\frac{1}{\varepsilon}\|\overline{Q}'\|_{L^{2}}^{2}+ \frac{\varepsilon}{\lambda_{1}}\int_{M}u\overline{P}'u\notag.
\end{align}
Thus, for $\varepsilon=\frac{\lambda_{1}}{2}$,
$$\frac{1}{2}\int_{M}u\overline{P}'u\leq E(u_{0})-2\bar{u}\int_{M}\overline{Q}'+\frac{2}{\lambda_{1}}\|\overline{Q}'\|_{L^{2}}^{2}.$$
Therefore
$$\int_{M}e^{2u}\leq C\exp\left(\frac{1}{8\pi^{2}}E(u_{0})+ \frac{\|\overline{Q}'\|_{L^{2}}^{2}}{4\lambda_{1}\pi^{2}}+\left(2-\frac{1}{4\pi^{2}}\int_{M}\overline{Q}'\right)\bar{u}\right) .$$
Now we notice that $E(u_{0})\leq \|u_{0}\|^{2}+\|\overline{Q}'\|_{L^{2}}^{2}$, hence there exist constants $C_{1}$ and $C_{2}$ such that
$$\int_{M}e^{2u}\leq C_{1}\exp\left(\frac{1}{8\pi^{2}}\|u_{0}\|^{2}+C_{2}\int_{M}u\right).$$
By using Lemma \ref{concentration1}, we get
$$\int_{M}e^{2u}\leq \bar C_{K}\exp\left(A_{1}\|u_0\|^{2}+\frac{A_{2}}{Vol(K)}\max\left(\int_{K}u,0\right)\right),$$
where $\bar C_{K}$ depends on $Vol(K)$. Now, we set $\alpha=\max\left(A_{1},\frac{A_{2}}{2},2\right)>1$, and we get
$$\int_{M}e^{2u}\leq \bar C_{K}\exp\left( \alpha \|u_0\|^{2}+\frac{\alpha}{Vol(K)}\max\left(\int_{K}2u,0\right)\right).$$
But Jensen's inequality yields
$$\exp\left(\frac{1}{Vol(K)}\displaystyle\int_{K}2u\right)\leq \frac{1}{Vol(K)}\int_{K}e^{2u},$$
in particular
$$\exp\left(\frac{\alpha}{Vol(K)}\max\left(\int_{K}u,0\right)\right)\leq \max\left(\left(\frac{1}{Vol(K)}\int_{K}e^{2u}\right)^{\alpha},1\right).$$
Therefore, by adjusting the constant eventually
$$\int_{M}e^{2u}\leq C_{K}e^{\alpha\|u_{0}\|^{2}}\max\left(\left(\int_{K}e^{2u}\right)^{\alpha},1\right)$$
which completes the proof.
\end{proof}

\noindent
Next we move to the proof of Theorem \ref{negativecase}. We set
$$K=\left\{x\in M; f(x)\leq \frac{1}{2}\inf_{x\in M}f(x)\right\}.$$
From the compatibility condition $(i)$ in (\ref{hypotheses}), we have that $Vol(K)>0$, and since
$$\int_{M}\overline{Q}'=\int_{M}fe^{2u_{0}} ,$$
we have
$$\frac{\int_{M}\overline{Q}'}{\displaystyle\inf_{x\in M} f(x)}\leq \int_{M}e^{2u_{0}}.$$
Thus, there exists $C>0$ (we will assume $C>1$ actually) such that
$$\int_{M}e^{2u_{0}}\leq C\exp\left[C\left(\int_{M}u_{0}\overline{P}'u_{0} +\int_{M}u_{0}^{2}\right)\right]=Ce^{C\|u_{0}\|^{2}}.$$
Hence,
\begin{equation}\label{formula1}
  \frac{\int_{M}\overline{Q}'}{\displaystyle\inf_{x\in M} f(x)}\leq Ce^{C\|u_{0}\|^{2}}.
\end{equation}

\noindent
Next we will prove the following
\begin{lemma}\label{concentration3}
Let $C_{K}$ and $\alpha$ be the constants found in Lemma \ref{concentration2}. Let
$$r=C_{K}(8C)^{\alpha}e^{(C+1)\alpha \|u_{0}\|^{2}}, $$
and let us assume that
$$e^{\tau \|u_{0}\|^{2}}\sup_{x\in M} f(x)\leq C_{0},$$
where $\tau=\alpha(C+1)-C$ and
$$C_{0}=-\frac{\displaystyle\inf_{x\in M} f(x)}{8^{\alpha}C_{K}C^{\alpha-1}}.$$
Then for all $t\geq 0$, it holds
$$\int_{M}e^{2u}\leq 2r.$$
\end{lemma}

\begin{proof}
Let
$$T=\sup\left\{s\geq 0; \int_{M}e^{2u} \leq 2r \text{ in } [0,s]\right\}$$
and let us assume for the sake of contradiction that $T<\infty$. We notice that by continuity, we have that
$$\int_{M}e^{2u(T)}=2r.$$
We assume first that
$$\int_{M}f^{+}e^{2u(T)} \leq \frac{1}{2}\int_{M}f^{-}e^{2u(T)},$$
where $f^+:=\max\{f,0\}$ and $f^-=f^{+}-f$ denote the positive and negative part of $f$ respectively. Then we have
$$\int_{M}f^{-}e^{2u(T)}\leq -2\int_{M}fe^{2u(T)}=-2\int_{M}\overline{Q}'\leq -4\int_{M}\overline{Q}' .$$
Since in $K$ we have $f^{-}(x)\geq -\frac{1}{2}\displaystyle\inf_{x\in M} f(x)$, we have
$$\int_{K}e^{2u(T)}\leq \frac{8\int_{M}\overline{Q}'}{\displaystyle\inf_{x\in M} f(x)}$$
which combined with (\ref{formula1}) gives
$$\int_{K}e^{2u(T)}\leq 8C e^{C\|u_{0}\|^{2}}.$$
But from Lemma \ref{concentration2}, we have
$$\int_{M}e^{2u(T)}\leq C_{K}e^{\alpha \|u_{0}\|^{2}}\max\left(\left(\int_{K}e^{2u}\right)^{\alpha},1\right).$$
Thus
$$\int_{M}e^{2u(T)}\leq C_{K}e^{\alpha \|u_{0}\|^{2}}\left(8Ce^{C\|u_{0}\|^{2}}\right)^{\alpha}=r,$$
which is a contradiction.\\
So we move to the next case, where
$$\int_{M}f^{+}e^{2u(T)}>\frac{1}{2}\int_{M}f^{-}e^{2u(T)}.$$
Then we have
$$-\frac{1}{2}\inf_{x\in M} f(x)\int_{K}e^{2u(T)}\leq \int_{M}f^{-}e^{2u(T)}<2\int_{M}f^{+}e^{2u(T)}\leq 4r \sup_{x\in M} f(x).$$
Hence,
$$\int_{K}e^{2u(T)}\leq -\frac{8r\displaystyle\sup_{x\in M} f(x)}{\displaystyle\inf_{x\in M} f(x)}.$$
By using our assumption, we have that
$$ \int_{K}e^{2u(T)}\leq -\frac{8r e^{-\tau \|u_{0}\|^{2}}C_{0}}{\displaystyle\inf_{x\in M} f(x)},$$
and by Lemma \ref{concentration2}, we have
$$ \int_{M}e^{2u(T)}\leq C_{K}e^{\alpha\|u_{0}\|^{2}} \left(\frac{8r e^{-\tau \|u_{0}\|^{2}}C_{0}}{-\displaystyle\inf_{x\in M} f(x)}\right)^{\alpha}\leq r,$$
leading again to a contradiction. Hence $T=+\infty$ and $\displaystyle\int_{M}e^{2u}$ is uniformly bounded.
\end{proof}

\noindent
Now, by Jensen's inequality we have
$$\exp{\left(\frac{1}{V}\int_{M}2u\right)}\leq \frac{1}{V}\int_{M}e^{2u}\leq \frac{2r}{V},$$
thus $\bar{u}$ is bounded from above. Now again using the energy identity (\ref{energy identity}), we have
$$\int_{M}u\overline{P}'u+2\int_{M}\overline{Q}'(u-\bar{u})+2\bar{u}\int_{M}\overline{Q}'\leq E(u_{0}),$$
and
$$\int_{M}u\overline{P}'u+2\int_{M}\overline{Q}'(u-\bar{u})\geq \frac{1}{2}\int_{M}u\overline{P}'u -\frac{2\|\overline{Q}'\|^{2}_{L^{2}}}{\lambda_{1}}\geq -C_{3}.$$
Therefore
$$2\bar{u}\int_{M}\overline{Q}'\leq E(u_{0})+C_{3},$$
and since $\displaystyle\int_{M}\overline{Q}'<0$ we have that $\bar{u}$ is uniformly bounded from below which finishes the proof of Theorem \ref{negativecase}.


\subsection{Case $\displaystyle\int_{M}\bQ =0$ and proof of Theorem \ref{zerocase}}

\noindent
Since $\displaystyle\int_{M}\overline{Q}'=0$, we have that
$$\langle \nabla E(u),1\rangle=2\int_{M}\overline{P}'u=0$$
and
$$\langle \nabla N(u),1\rangle =2\int_{M} \Gamma\left(fe^{2u}\right)=2\int_{M}\overline{Q}'=0$$
Hence,
$$0=\int_{M}\partial_{t}u=\partial_{t}\int_{M}u,$$
which means that the average value of $u$ is preserved. Therefore $\bar{u}=\bar{u}_{0}$ and by Corollary \ref{average}, we have the convergence of the flow. This ends the proof of Theorem \ref{zerocase}.


\subsection{Case $0<\displaystyle\int_{M}\bQ <16\pi^{2}$ and proof of Theorem \ref{positivecase}}

\noindent
First, we have again from the energy identity (\ref{energy identity})
\begin{equation}\label{positive1}
  \int_{M}u\overline{P}'u+2\int_{M}\overline{Q}'(u-\bar{u})+2\bar{u}\int_{M}\overline{Q}' \leq E(u_{0}).
\end{equation}
Hence
$$2\bar{u}\int_{M}\overline{Q}'\leq E(u_{0})-\frac{1}{2}\int_{M}u\overline{P}'u +\frac{2}{\lambda_{1}}\|\overline{Q}'\|_{L^{2}}^{2}$$
and then $\bar{u}$ is bounded from above; we will need a bound from below. Since $u\in X$, we get
$$\int_{M}\overline{Q}'=\int_{M}fe^{2u}\leq \|f\|_{\infty}\int_{M}e^{2u} ,$$
and therefore
$$\ln\left(\frac{\displaystyle\int_{M}\overline{Q}'}{\|f\|_{\infty}}\right)\leq \ln\left(\int_{M}e^{2u}\right).$$
Now again from Theorem \ref{thmbran} we have
\begin{equation}\label{positive2}
\ln\left(\frac{\displaystyle\int_{M}\overline{Q}'}{\|f\|_{\infty}}\right)\leq C+\frac{1}{16\pi^{2}}\int_{M}u\overline{P}'u +\frac{2}{V}\int_{M}u.
\end{equation}

\noindent
Let $\delta>0$ to be determined later, we sum equation (\ref{positive1}) and $-\delta$ times equation (\ref{positive2}), obtaining
$$\ln\left(\frac{\displaystyle\int_{M}\overline{Q}'}{\|f\|_{\infty}}\right)-\delta E(u_{0})\leq C+\left(\frac{1}{16\pi^{2}}-\delta\right)\int_{M}u\overline{P}'u+2\left(1-\delta\int_{M}\overline{Q}'\right)\bar{u}
-2\delta\int_{M}\overline{Q}'(u-\bar{u}) .$$
Since $\displaystyle\int_{M}\overline{Q}'<16\pi^{2}$, we choose $\delta$ such that $\displaystyle\int_{M}\overline{Q}'<\frac{1}{\delta}<16\pi^{2}$, and we set
$$c_{1}=2\left(1-\delta \int_{M}\overline{Q}'\right), \qquad c_{2}=\delta-\frac{1}{16\pi^{2}}.$$
We have
$$\ln\left(\frac{\displaystyle\int_{M}\overline{Q}'}{\|f\|_{\infty}}\right) -\delta E(u_{0})-C+c_{2}\int_{M}u\overline{P}'u+2\delta\int_{M}\overline{Q}'(u-\bar{u})\leq c_{1}\bar{u}.$$
Now we notice that
$$c_{2}\int_{M}u\overline{P}'u+2\delta\int_{M}\overline{Q}'(u-\bar{u})\geq (c_{2}\lambda_1- \delta \varepsilon) \|u-\bar{u}\|_{L^{2}}^{2} -\frac{\delta}{\varepsilon}\|\overline{Q}'\|_{L^{2}}^{2},$$
therefore for $\varepsilon$ small enough we have that
$$c_{2}\int_{M}u\overline{P}'u+2\delta\int_{M}\overline{Q}'(u-\bar{u})\geq-c_{3} .$$
It follows that $\bar{u}$ is bounded from below and therefore from Corollary \ref{average} this finishes the proof.


\section{The critical case and proof of Theorem \ref{criticalcase}}

\noindent
Here we will study the case $\displaystyle\int_{M}\bQ =16\pi^{2}$, where $M=S^3$ is the sphere equipped with its standard contact structure. We will see $S^{3}$ as a subset of $\C^{2}$ with coordinates $(\zeta_{1},\zeta_{2})$ such that
$$S^3=\left\{(\zeta_1,\zeta_2)\in \C^2 \; ; \;|\zeta_{1}|^{2}+|\zeta_{2}|^{2}=1 \right\}.$$
Recall that the set $Aut(\H^1)$ of conformal transformation of the Heisenberg group $\H^1$ is generated by left translations, dilations. rotations and inversion. Using the Cayley transform $C:\H^{1}\to S^{3}\{(0,-i)\}$ one has a clear description of the set $Aut(S^{3})$:
$$Aut(S^{3})=\{C\circ h\circ C^{-1}; h\in Aut(\H^1)\}.$$
For $p\in S^{3}$ and $r\geq 1$, we will write $h_{p,r}$ the element of $Aut(S^{3})$ corresponding to a Cayley transform centered at $p$ and a dilation of size $r$. Now, for $u\in X$ we set
$$v_{p,r}=u\circ h_{p,r}+\frac{1}{2}\ln(J(h_{p,r})),$$
where we denoted $J(h)=det(Jac(h))$, the Jacobian determinant of $h$.
We have
$$E(v_{p,r})=E(u)\leq E(u_{0}),$$
and since $u\in X$
$$\int_{S^{3}}f\circ h_{p,r}e^{2v_{p,r}}=\int_{S^{3}}fe^{2u},$$
hence
$$\int_{S^{3}}e^{2v_{p,r}}\geq \frac{16\pi^{2}}{\displaystyle\sup_{x\in S^{3}} f(x)}.$$
From \cite[page~38]{Bran}, we know that for all $t\geq 1$ there exists $r(t)\geq 1$ and $p(t)\in S^3$ such that
$$\int_{S^{3}}\xi_{i}e^{2v_{p(t),r(t)}}=0, \; i=1,2.$$
So we let $v(t)=v_{p(t),r(t)}$ and $h(t)=h_{p(t),r(t)}$. Then using Corollary \ref{imp} in the Appendix, one has the existence of $a<\displaystyle\frac{1}{16\pi^{2}}$ and a constant $C_1$ such that
$$a\int_{S^{3}}v(t)\overline{P}'v(t)+2\int_{S^{3}}v(t)- \ln\left(\int_{S^{3}}e^{2v(t)}\right)+C_1\geq 0.$$
Since $E(v(t))\leq E(u_{0})$, we find that
$$\int_{S^{3}}v(t)\overline{P}'v(t)\leq C,$$
and
$$\left|\int_{S^{3}}v(t)\right| \leq C.$$
In particular we have that for all $p\geq 1$
$$\int_{S^{3}}e^{2|pv(t)|}\leq C_{p},$$
and hence
$$\int_{S^{3}}v^{2}(t)\leq C$$
leading to the boundedness of $v(t)$ in $H$. We need the following concentration-compactness lemma in order to prove uniform boundedness.
\begin{lemma}\label{Concentration-Compactness}
Either
\item
$(i)\quad \|u(t)\|\leq C$, for some constant $C$;\\
or
\item
$(ii)\quad $ there exists a sequence $t_{n}\to \infty$ and a point $p_{0}\in S^{3}$ such that for all $r>0$
$$\lim_{n\to \infty}\int_{B_{r}(p_{0})}fe^{2u(t_{n})}=16\pi^{2}.$$
Moreover, for any $\tilde{x}\in S^{3}\setminus \{p_{0}\}$, and any $r<d(\tilde{x},p_{0})$, we have
$$\lim_{n\to \infty}\int_{B_{r}(\tilde{x})}fe^{2u(t_{n})}=0.$$
\end{lemma}

\begin{proof}
We assume first that $r(t)$ is bounded. Then we have that
$$0<C_{1}\leq J(h_{p(t),r(t)})\leq C_{2}.$$
Thus, from the uniform boundedness of $v(t)$ we have
$$\int_{S^{3}}|u(t)|\leq C.$$
Therefore, from Lemma \ref{average}, it follows that $\|u(t)\|$ is uniformly bounded.\\
So now we assume that $r(t)$ is not bounded, then there exists a sequence $t_{n}\to \infty$ such that $r(t_{n})\to \infty$ and without loss of generality, by compactness of $S^{3}$ we can assume that $p(t_{n})\to p_{0}$. From the uniform boundedness of $v(t)$, we can also assume that $v(t_{n})\to v_{\infty}$ strongly in $L^{2}(S^{3})$ and weakly in $H$. We let then $r>0$ and set $K_{n}=h(t_{n})^{-1}(B_{r}(p_{0}))$. Then we have
$$\left|\int_{S^{3}}f\circ h(t_{n})e^{2v(t_{n})}-\int_{K_{n}}f\circ h(t_{n})e^{2v(t_{n})}\right|\leq \left(\sup_{x\in S^{3}} f(x)\right) \left(Vol(K_{n}^{c})\int_{S^{3}}e^{4|v(t_{n})|}\right)^{\frac{1}{2}}.$$
Since $h(t_{n})(x)\to p_{0}$ a.e. then $\displaystyle\lim_{n\to \infty}Vol(K_{n})=V,$ and thus
$$\int_{B_{r}(p_{0})}fe^{2u(t_{n})}=\int_{K_{n}}f\circ h(t_{n})e^{2v(t_{n})}=\int_{S^{3}}f\circ h(t_{n})e^{2v(t_{n})}+o(1).$$
We have also
$$\int_{S^{3}}f\circ h(t_{n})e^{2v(t_{n})}=16\pi^{2},$$
and then
$$\lim_{n\to \infty}\int_{B_{r}(p_{0})}fe^{2u(t_{n})}=16\pi^{2}.$$
Now if we consider $\tilde{x}\in S^{3}\setminus \{p_{0}\}$ and $r<d(p_{0},\tilde{x})$ we have that $h(t_{n})(x)\not \in B_{r}(\tilde{x})$ for $n$ big enough, since $\displaystyle\lim_{n\to \infty} h(t_{n})(x)=p_{0}$ a.e.; in particular
$$\lim_{n\to \infty}\chi_{h(t_{n})^{-1}(B_{r}(\tilde{x}))}=0,$$
where $\chi$ is the characteristic function. Therefore
$$\lim_{n\to \infty}\int_{B_{r}(\tilde{x})}fe^{2u(t_{n})}=\lim_{n\to \infty}\int_{h(t_{n})^{-1}(B_{r}(\tilde{x}))}f\circ h(t_{n})e^{2v(t_{n})}=0.$$
\end{proof}

\noindent
Let us assume now that $\Sigma = \emptyset$. By using the previous lemma, if $\|u(t)\|$ is not uniformly bounded, then there exists $p_{0}\in S^{3}$ such that
$$\lim_{n\to \infty}\int_{B_{r}(p_{0})}fe^{2u(t_{n})}=16\pi^{2},$$
and if $p_{1}\not=p_{0}$ and $r<d(p_{0},p_{1})$, then
$$\lim_{n\to \infty}\int_{B_{r}(p_{1})}fe^{2u(t)}=0.$$
Since $\Sigma=\emptyset$, then there exists $g\in G$ such that $p_{1}=g\cdot p_{0}\not=p_{0}$. But
$$16\pi^{2}=\lim_{n\to \infty}\int_{B_{r}(p_{0})}fe^{2u(t)}=\lim_{n\to \infty}\int_{B_{r}(g\cdot p_{0})}fe^{2u(t)}=\lim_{n\to \infty}\int_{B_{r}(p_{1})}fe^{2u(t)}=0$$
which is a contradiction. Hence $\|u(t)\|$ is uniformly bounded.\\
Now we assume that $\Sigma\neq\emptyset$ and that $\|u(t)\|$ is not uniformly bounded. We have that the concentration point $p_{0}\in \Sigma$, otherwise we reach a contradiction arguing as in the previous case. So we have
$$\int_{B_{r}(p_{0})}fe^{2u(t_{n})}\leq \sup_{x\in B_{r}(p_{0})}f(x)\int_{B_{r}(p_{0})}e^{2u(t_{n})}\leq \max\left(\sup_{x\in B_{r}(p_{0})}f(x),0\right)\int_{S^{3}}e^{2u(t_{n})}.$$
By using the the sphere version of Theorem \ref{thmbran}, proved in \cite{Bran}, we have that
$$\frac{1}{V}\int_{S^{3}}e^{2u(t_{n})}\leq e^{\frac{E(u(t_{n}))}{V}}.$$
Thus
$$\int_{B_{r}(p_{0})}fe^{2u(t_{n})}< \max\left(\sup_{x\in B_{r}(p_{0})}f(x),0\right)V e^{\frac{E(u_{0})}{V}}.$$
Now we first let $n\to \infty$, then $r\to \infty$ and we get
$$16\pi^{2}< V \max(f(p_{0}),0)e^{\frac{E(u_{0})}{V}}.$$
Therefore $f(p_{0})>0$ and
$$1< f(p_{0})e^{\frac{E(u_{0})}{16\pi^{2}}},$$
hence
$$f(p_{0})>e^{-\frac{E(u_{0})}{16\pi^{2}}},$$
which leads to a contradiction of the assumption in Theorem \ref{criticalcase}. Therefore we have the uniform boundedness of $\|u\|$ also in this case, which yields the convergence of the flow and it ends the proof.



\appendix
\section{Appendix: Improved Moser-Trudinger Inequality }

\noindent
In what follows we will consider $S^{3}$ as a subset of $\C^{2}$ with coordinates $(\zeta_{1},\zeta_{2})$ such that $|\zeta_{1}|^{2}+|\zeta_{2}|^{2}=1$. We recall here the Improved Moser-Trudinger inequality introduced in \cite{Bran} in order to prove the existence of a minimizer:

\begin{proposition}(\cite{Bran}, Proposition 3.4)\label{propbran}
Given $\frac{1}{2}<a<1$, there exist constants $C_{1}(a)$, $C_{2}(a)$ such that for $u\in H$ with $\int_{S^{3}}\zeta_{i}e^{2u}=0$, $i=1,2$, it holds:
$$\frac{a}{16\pi^{2}}\int_{S^{3}}uP'u+2\int_{S^{3}}u -\ln\left(\int_{S^{3}}e^{2u}\right)+C_{1}(a)\|(-\Delta_{b})^{\frac{3}{4}}u\|_{2}^{2}+C_{2}(a)\geq 0$$
\end{proposition}

\noindent
This improved estimate will not be useful to us in our setting since it contains the term $C_{1}(a)\|(-\Delta_{b})^{\frac{3}{4}}u\|_{2}^{2}$ that we cannot bound along the flow. Notice that in \cite{Bran}, the authors exploit Ekeland's principle to exhibit a good minimizing Palais-Smale sequence that allows the control of this extra term. In our setting, we will prove a result that can be seen as intermediate between Proposition \ref{propbran} and the usual Moser-Trudinger inequality in Theorem \ref{thmbran}. In fact in \cite{Bran} the authors gave hints on how to prove this result, knowing that this method only works in dimension 3 and 5. We will follow a technique used in \cite{CB}, since it is simpler and it allows even more improved estimates.

\noindent
We set
$$P_{k}:=\Big\{\text{polynomials of }\C^{2}\text{ with degree at most }k\Big\}$$
and
$$P_{k,0}:=\Big\{f\in P_{k};\int_{S^{3}}f=0\Big\}.$$
For a given $m\in \N$ we let
$$\mathcal{N}_{m}:=\left\{\begin{array}{lc}
N\in \N; \exists x_{1},\cdots,x_{N}\in S^{3},\nu_{1},\cdots,\nu_{N}\in \R^{+}
\text{ with } \sum_{k=1}^{N}\nu_{k}=1\\
\text{ and for any } f\in P_{m,0}; \sum_{k=1}^{N}\nu_{k}f(x_{k})=0
\end{array}
\right\}.$$
We let then $N_{m}=\min \mathcal{N}_{m}$. As it was shown in \cite{CB}, one has $N_{1}=2$ and $N_{2}=4$. 
We recall from \cite{Bran} that one has the following inequality on the standard sphere:\\
There exists a constant $A_{2}>0$ such that
$$\int_{S^{3}}\exp\Big[A_{2}\frac{|u-\bar{u}|^{2}}{\|\Delta_{b} u\|_{L^{2}}^{2}}\Big]\leq C_{0}.$$
In fact the sharp constant $A_{2}$ was explicitly computed in \cite{Bran} and it has the value $A_{2}=32$. With this result we can easily deduce that if $u\in \mathcal{S}^{2}(S^{3})$ then $e^{2u}\in L^{p}(S^{3})$ for all $1\leq p< \infty$.
\begin{lemma}\label{lemA1}
Consider a sequence of functions $u_{k}\in \mathcal{S}^{2}(S^{3})$ such that
$$\bar{u}_{k}=0, \quad \|\Delta_{b}u_{k}\|_{L^{2}}\leq 1$$
and suppose that $u_{k}\rightharpoonup u$ weakly in $\mathcal{S}^{2}(S^{3})$ and
$$|\Delta_{b}u_{k}|^{2}\rightharpoonup |\Delta_{b}u|^{2} +\sigma \text{ in measure },$$
where $\sigma$ is a measure on $S^{3}$. Let $K\subset S^{3}$ be a compact set with $\sigma(K)<1$, then for all $1\leq p <\frac{1}{\sigma(K)}$ we have
$$\sup_{k}\int_{K}\exp\Big[pA_{2}u_{k}^{2}\Big] <\infty.$$
\end{lemma}

\begin{proof}
Let $\varphi$ be a fixed smooth compactly supported function on $S^{3}$. We set $v_{k}=u_{k}-u$. Then $v_{k}\to 0$ strongly in $L^{2}$ and weakly in $\mathcal{S}^{2}(S^{3})$. Now we compute
\begin{align}
\int_{S^{3}}|\Delta_{b}(\varphi v_{k})|^{2}&=\int_{S^{3}}(\varphi \Delta_{b}v_{k}+v_{k}\Delta_{b}\varphi+2\nabla_{H}\varphi \nabla_{H}v_{k})^{2}\notag\\
&=\int_{S^{3}}\varphi^{2}(\Delta_{b}v_{k})^{2}+v_{k}^{2}(\Delta_{b}\varphi)^{2}+4|\nabla_{H}v_{k}\nabla_{H}\varphi|^{2}+2\varphi v_{k}\Delta_{b}\varphi \Delta_{b}v_{k}+\\
&\qquad+4\varphi(\nabla_{H}\varphi \nabla_{H}v_{k})\Delta_{b}v_{k}+4v_{k}(\nabla_{H}v_{k}\nabla_{H}\varphi) \Delta_{b}\varphi. \notag
\end{align}
Hence,
$$\int_{S^{3}}|\Delta_{b}(\varphi v_{k})|^{2}\to \int_{S^{3}}\varphi^{2}d\sigma.$$
Assume that $1\leq p_{1}<\frac{1}{\sigma(K)}$ and take $\varphi$ so that $\varphi_{|K}=1$, and $\int_{S^{3}}\varphi^{2}d\sigma<\frac{1}{p_{1}}$. Then we have for $k$ large,
$$\|\Delta_{b}(\varphi v_{k})\|_{L^{2}}^{2}<\frac{1}{p_{1}}.$$
Therefore,
$$\int_{K}\exp\Big[p_{1}A_{2}(v_{k}-\overline{\varphi v_{k}})^{2}\Big] \leq \int_{S^{3}}\exp\Big[p_{1}A_{2}(\varphi v_{k}-\overline{\varphi v_{k}})^{2}\Big] \leq\int_{S^{3}}\exp\Big[A_{2}\frac{(\varphi v_{k}-\overline{\varphi v_{k}})^{2}}{\|\Delta_{b}\varphi v_{k}\|_{L^{2}}^{2}}\Big]\leq C_{0}.$$
Thus, if we fix $\varepsilon>0$, we can write
\begin{align}
u_{k}^{2}&=(v_{k}-\overline{\varphi v_{k}}+u+\overline{\varphi v_{k}})^{2}\notag\\
&=(v_{k}-\overline{\varphi v_{k}})^{2}+2(v_{k}-\overline{\varphi v_{k}})(u+\overline{\varphi v_{k}})+(u+\overline{\varphi v_{k}})^{2}\notag\\
&\leq (1+\varepsilon)(v_{k}-\overline{\varphi v_{k}})^{2}+2(1+\frac{1}{\varepsilon})u^{2}+2(1+\frac{1}{\varepsilon})^{2}\overline{\varphi v_{k}}^{2}.\notag
\end{align}
Hence, given $p<\frac{1}{\sigma(K)}$ we can take $p_{1}\in (p,\frac{1}{\sigma(K)})$ such that
$$\int_{K}e^{A_{2}p_{1}u_{k}^{2}}<C_{0},$$
which finishes the proof.
\end{proof}

\begin{corollary}
We consider the same assumptions as in Lemma \ref{lemA1} and we let $\ell=\max_{x\in S^{3}}\sigma(\{x\})\leq 1.$ Then the following hold
\begin{itemize}
\item If $\ell<1$, then for any $1\leq p<\frac{1}{\ell}$, $e^{A_{2}u_{k}^{2}}$ is bounded in $L^{p}(S^{3})$. In particular $e^{A_{2}u_{k}^{2}}\to e^{A_{2}u^{2}}$ in $L^{1}$.
\item If $\ell=1$, then there exists $x_{0}\in S^{3}$ such that $\sigma=\delta_{x_{0}}$, $u=0$ and after passing to a subsequence if necessary, we have
$$e^{A_{2}u_{k}^{2}}\rightharpoonup 1+c_{0}\delta_{x_{0}},$$
for some $c_{0}\geq0$.
\end{itemize}
\end{corollary}

\begin{proof}
Assume that $\ell<1$ and let $1\leq p<\frac{1}{\ell}$. Then for all $x\in S^{3}$, $\sigma(\{x\})<\frac{1}{p}$. By continuity, there exists $r_{x}>0$ such that $\sigma(\overline{B_{r_{x}}(x)})<\frac{1}{p}$. Since $S^{3}$ is compact we can find a finite collection of balls of the form $B_{r_{i}}(x_{i})$ such that
$$S^{3}=\bigcup_{i=1}^{N}\overline{B_{r_{i}}(x_{i})}.$$
So using Lemma \ref{lemA1}, we have
$$\sup_{k}\int_{\overline{B_{r_{i}}(x_{i})}}\exp\Big[pA_{2}u_{k}^{2}\Big]<\infty.$$
Thus,
$$\sup_{k}\int_{S^{3}}\exp\Big[pA_{2}u_{k}^{2}\Big]<\infty.$$
We assume now that $\ell=1$. Since $\|\Delta_{b}u_{k}\|^{2}\leq 1$ we have that
$\|\Delta_{b}u\|^{2}+\sigma(S^{3})\leq 1$. Therefore, we have $u=0$ and there exists $x_{0}\in S^{3}$ such that $\sigma=\delta_{x_{0}}$. Hence, for $r$ small, we have that
$$\sup_{k}\int_{S^{3}\setminus B_{r}(x_{0})}\exp\Big[qA_{2}u_{k}^{2}\Big]<\infty,$$
for all $q\geq 1$. Therefore, $e^{A_{2}u_{k}^{2}}\to 1$ in $L^{1}(S^{3}\setminus B_{r}(x_{0}))$ for every $r>0$ and small. Hence, after passing to a subsequence if necessary we have that $e^{A_{2}u_{k}^{2}}\rightharpoonup 1+c_{0}\delta_{x_{0}}$ in measure.
\end{proof}

\begin{proposition}\label{propA1}
Let $\alpha>0$ and consider a sequence $m_{k}\to \infty$ and $u_{k}\in \mathcal{S}^{2}(S^{3})$ such that $\overline{u_{k}}=0$ and $\|\Delta_{b}u_{k}\|_{L^{2}}=1$ such that $u_{k}\rightharpoonup u$ weakly in $\mathcal{S}^{2}(
S^{3})$ and $(\Delta_{b}u_{k})^{2}\rightharpoonup (\Delta_{b}u)^{2}+\sigma$ in measure. We assume moreover that
$$\ln \left(\int_{S^{3}}e^{2m_{k}u_{k}}\right)\geq \alpha m_{k},$$
and
$$\frac{e^{2m_{k}u_{k}}}{\displaystyle\int_{S^{3}}e^{2m_{k}u_{k}}}\rightharpoonup \nu \text{ in measure}.$$
We set $R=\Big\{x\in S^{3}; \sigma(\{x\})\geq A_{2}\alpha\Big\}=\{x_{1},\cdots,x_{N}\}$.
Then $\nu=\sum_{i=1}^{N}\nu_{i}\delta_{x_{i}}$ with $\nu_{i}\geq 0$ and $\sum_{i}\nu_{i}=1$.
\end{proposition}

\begin{proof}
Let $K\subset S^{3}$ such that $\sigma(K)<A_{2}\alpha$. By continuity, we can find a compact set $K_{1}$ such that $K\subset int(K_{1})$ and $\sigma(K_{1})<A_{2}\alpha$. Now given $\frac{1}{A_{2}\alpha}<p<\frac{1}{\sigma(K_{1})}$, we have
$$\sup_{k}\int_{K_{1}}e^{pA_{2}u_{k}^{2}}\leq C_{0}.$$
Since $2m_{k}u_{k}\leq pA_{2}u_{k}^{2}+\frac{m_{k}^{2}}{pA_{2}}$, we have
$$\int_{K_{1}}e^{2m_{k}u_{k}}\leq Ce^{\frac{m_{k}^{2}}{A_{2}p}}.$$
Therefore,
$$\frac{\displaystyle\int_{K_{1}}e^{2m_{k}u_{k}}}{\displaystyle\int_{S^{3}}e^{2m_{k}u_{k}}}\leq Ce^{\Big(\frac{1}{A_{2}p}-\alpha\Big)m_{k}^{2}}.$$
So $\nu(K)\leq \nu(K_{1})=0$ and $\nu(K)=0$. Thus, if $\sigma(\{x\})<A_{2}\alpha$, then there exists $r_{x}>0$ small enough so that $\sigma(\overline{B_{r_{x}}(x)})<A_{2}\alpha$. Hence, $\nu(\overline{B_{r_{x}}(x)})=0$. We deduce then that $\nu(S^{3}\setminus R)=0.$\\
Therefore
$$\nu=\sum_{k=1}^{N}\nu_{k}\delta_{x_{k}},$$
with $\nu_{k}\geq 0$ and $\sum_{k=1}^{N}\nu_{k}=1$.
\end{proof}

\noindent
Let $f_{1},\cdots, f_{\ell} \in C(S^{3})$. We define
$$\mathcal{S}_{f}=\left\{u\in \mathcal{S}^{2}(S^{3});\overline{u}=0; \int_{S^{3}}f_{k}e^{2u}=0; k=1,\cdots,\ell\right\}.$$
We assume that the inequality
$$\ln\left(\int_{S^{3}}e^{2u}\right)\leq \alpha\|\Delta_{b}u\|^{2}_{2}+C$$
does not hold for $u\in\mathcal{S}_{f}$. Then there exists a sequence $u_{k}\in \mathcal{S}_{f}$ such that
$$\ln\left(\int_{S^{3}}e^{2u_{k}}\right)-\alpha\|\Delta u_{k}\|_{L^{2}}^{2}\to \infty.$$
Therefore, it follows that $\int_{S^{3}}e^{2u_{k}} \to \infty$ and $\|\Delta_{b}u_{k}\|_{L^{2}}\to \infty$. So we let $m_{k}=\|\Delta_{b}u_{k}\|_{L^{2}}$ and $v_{k}=\frac{u_{k}}{m_{k}}$. Then $m_{k}\to \infty$, $\|\Delta_{b} v_{k}\|^{2}_{L^{2}}=1$. Hence, after passing to a subsequence, we have
$$\left\{\begin{array}{lll}
v_{k}\rightharpoonup v \text{ weakly in } \mathcal{S}^{2}(S^{3}),\\
\\
|\Delta_{b}v_{k}|^{2}\rightharpoonup |\Delta_{b}v|^{2}+\sigma \text{ in measure},\\
\\
\frac{e^{2m_{k}v_{k}}}{\displaystyle\int_{S^{3}}e^{2m_{k}v_{k}}}\rightharpoonup \nu \text{ in measure}.
\end{array}
\right.$$
So we let $R=\{x\in S^{3}; \sigma(\{x\})\geq A_{2}\alpha\}=\{x_{1},\cdots,x_{N}\}$. It follows from Proposition \ref{propA1} that $\nu=\sum_{j=1}^{N}\nu_{j}\delta_{x_{j}}$, with $\sum_{j=1}^{N}\nu_{j}=1$ and $\nu_{j}\geq 0$.\\
But since $u_{k}\in \mathcal{S}_{f}$, we have$$\int_{S^{3}}f_{j}d\nu=0.$$
Therefore,
$$\sum_{i=1}^{N}\nu_{i}f_{j}(x_{i})=0, \text{ for all } 1\leq j \leq \ell .$$
On the other hand, $A_{2}\alpha N\leq 1$. In particular, if $f_{j}\in P_{m,0}$, we have that $N\in \mathcal{N}_{m}$. Therefore,
$$\alpha\leq \frac{1}{A_{2}N}\leq \frac{1}{A_{2}N_{m}}.$$
Hence, if $\alpha=\frac{1}{A_{2}N_{m}}+\varepsilon$ we get a contradiction and the result holds. Therefore, if we define
$$\mathcal{S}_{0}=\left\{u\in \mathcal{S}^{2}(S^{3}); \overline{u}=0; \int_{S^{3}}fe^{2u}=0 \text{ for all } f\in P_{1,0}\right\},$$
the following corollary holds
\begin{corollary}\label{imp}
There exist $a<\frac{1}{16\pi^{2}}$ and $C>0$ such that for all $u\in \mathcal{P}\cap \mathcal{S}_{0}$, we have
$$a\int_{S^{3}}u\overline{P}'u+2\int_{M}u-\ln\left(\int_{M}e^{2u}\right) \geq -C.$$
\end{corollary}

\noindent
Indeed, this corollary follows from the fact that
$$\int_{S^{3}}u\overline{P}'u \geq \int_{S^{3}}|2\Delta_{b}u|^{2}$$
for all $u\in \hat{\P}$ and $8A_{2}>16\pi^{2}$.


\end{document}